\documentclass[11pt]{amsart}

\usepackage{amscd}
\usepackage{xypic}  
\usepackage{amssymb}
\usepackage{amsthm}
\usepackage{epsfig}
\usepackage{paralist}
\usepackage{comment}
\usepackage{url}
\usepackage{lipsum}   
\usepackage{tikz}

\usepackage[all,cmtip]{xy}
\usepackage{amsmath}
\usepackage{color}

\newtheorem{thm}{Theorem}[section]  

\newtheorem{lemma}[thm]{Lemma}

\newtheorem{proposition}[thm]{Proposition}

\newtheorem{corollary}[thm]{Corollary}

\theoremstyle{definition}

\newtheorem{remark}[thm]{Remark}

\def\ker{\operatorname{ker}}

\def\im{\operatorname{im}}

\def\c1{\operatorname{c_1}}
\def\c2{\operatorname{c_2}}

\def\rk{\operatorname{rk}}

\def\CC{{\mathbb C}}

\def\PP{{\mathbb P}}

\def\DD{{\mathbb D}}

\def\G{{\mathcal G}}
\def\L{{\mathcal L}}

\def\O{{\mathcal O}}

\def\E{{\mathcal E}}

\def\F{{\mathcal F}}
\def\K{{\mathcal K}}

\def\Q{{\mathcal Q}} 
\def\K{{\mathcal K}}

\def\c{\mathfrak{c}}

\def\UU{\mathfrak{U}}
\def\cong{\simeq}

\def\+{\oplus}               
\def\*{\otimes}     
\def\geq{\geqslant}
\def\leq{\leqslant}              

\def\id{\operatorname{id}}

\def\Ext{\operatorname{Ext}}

\def\Shext{\operatorname{ \mathfrak{e}\mathfrak{x}\mathfrak{t} }}

\def\det{\operatorname{det}}

\author[C.~Ciliberto]{Ciro Ciliberto}
\address{Ciro Ciliberto, Dipartimento di Matematica, Universit{\`a} di Roma Tor Vergata, Via della Ricerca Scientifica, 00173 Roma, Italy}
\email{cilibert@mat.uniroma2.it}

\author[F.~Flamini]{Flaminio Flamini}
\address{Flaminio Flamini, Dipartimento di Matematica, Universit{\`a} di Roma Tor Vergata, Via della Ricerca Scientifica, 00173 Roma, Italy} 
\email{flamini@mat.uniroma2.it}

\author[A.~L.~Knutsen]{Andreas Leopold Knutsen}
\address{Andreas Leopold Knutsen, Department of Mathematics, University of Bergen, Postboks 7800,
5020 Bergen, Norway}
\email{andreas.knutsen@math.uib.no}

\title{Ulrich bundles on a general blow--up of the plane}
\begin{document}

\maketitle

\begin{abstract} We prove that on $X_n$, the plane blown--up at $n$ general points, there are Ulrich line bundles with respect to  a line bundle corresponding to curves of degree $m$ passing simply through the $n$ blown--up points, with $m\leq 2\sqrt{n}$ and such that the line bundle in question is very ample on $X_n$.  We prove that the number of these Ulrich line bundles tends to infinity with $n$.   
 We also prove the existence of  slope--stable rank--$r$ Ulrich vector bundles on $X_n$,  for $n\geq 2$ and  any $r \geq 1$  and we compute the dimensions of their moduli spaces. These computations  imply  that $X_n$ is {Ulrich wild}. \end{abstract}

\section*{Introduction}

Let $X$ be a smooth irreducible projective variety of dimension $n \geqslant 1$ and let $H$ be a very ample divisor on $X$. A vector bundle $\E$  on $X$ is said to be an \emph{Ulrich vector bundle} with respect to the polarization $H$ if $h^i(\E(-pH))=0$ for all $i \geq 0$ and all $1 \leq p \leq n$. Ulrich bundles first appeared in commutative algebra in the 1980's, where they have been considered because they enjoy some extremal cohomological properties.  After that the attention of the algebraic geometers on these bundles has been carried by the beautiful paper \cite {ESW}, where, among other things, the authors compute the Chow form of a projective variety $X$ using Ulrich bundles on $X$, if they exist.  
In recent years there has been a good amount of work  on Ulrich bundles (for surveys see for instance \cite{Co, be2, CMP}), mainly investigating the following problems: given any polarization $H$ on a variety $X$, does there exist an Ulrich vector bundle with respect to $H$? Or, even more generally, given a variety $X$, does there exist a very ample line bundle $H$ on $X$
and an Ulrich vector bundle on $X$ with respect to $H$? What is the smallest possible rank for an Ulrich bundle on a given variety $X$? If Ulrich bundles exist, are they stable, and what are their moduli? Although a lot is known about these problems for some specific classes of varieties (curves, Segre, Veronese, Grassmann varieties, rational scrolls, complete intersections, some classes of surfaces like Del Pezzo, abelian, K3 surfaces, some surfaces of general type, etc.) the above   questions  are still open in their full  generality even for surfaces. In particular the question for which pairs $(X,H)$ there are Ulrich line bundles on $X$ with respect to $H$ is open, though this occurrence seems to be rather rare. Moreover in the few known cases in which they have been proved to exist, except the case of curves, they are finitely many and of a low number. 

In the present paper we investigate the existence of Ulrich bundles on the blow--up $X_n$ of the complex projective plane at $n$ very general points with respect to line bundles $\xi_{n,m}$ corresponding to curves of degree $m$ passing simply through the $n$ blown--up points, with $m\leq 2\sqrt{n}$ and such that the line bundle in question is very ample.
Surprisingly enough, our results show that such a surface carries several Ulrich line bundles,  and actually their number increases and tends to infinity with $n$, see Theorem \ref {thm:main1} and Corollary \ref {cor:inf}. In Theorem \ref {thm:class} we classify all these line bundles for $7\leq n\leq 10$  and $m=4$  (the cases $n\leq 6$ and $m=3$ corresponding to the Del Pezzo case being well known already, see \cite{CKM, PT}). Then, making iterated extensions and deforming,  we prove the existence of  slope-stable rank--$r$ Ulrich vector bundles on $X_n$ for  $n\geq 2$ and any  $r \geq 1$ (see Theorem \ref {thm:higher}), we compute the dimension of the moduli spaces of the bundles in question and prove that they are reduced. These computations show that $X_n$,  with $n\geq 2$,  is \emph{Ulrich wild} (recall that, as suggested by an analogous definition in \cite {DG}, a variety $X$ is said to be Ulrich wild if it possesses families of dimension $p$ of pairwise non--isomorphic, indecomposable, Ulrich bundles for arbitrarily large $p$). Note that in the literature there are only very few cases  of varieties known to carry stable Ulrich bundles of infinitely many ranks, and even fewer of 
\emph{any} rank (namely curves and Del Pezzo surfaces). Finally we mention that in \cite [Cor. 3.8]{ACM} the authors prove the existence of rank 2 Ulrich vector bundles on the blown--up planes with respect to an ample line bundle of the form $\xi_{n,m}$, without investigating stability or moduli.

\medskip

{\bf Acknowledgements:} Ciro Ciliberto and Flaminio Flamini are members of \linebreak GNSAGA of the Istituto Nazionale di Alta Matematica ``F. Severi"  and acknowledge support from  the MIUR Excellence Department Project awarded to the Department of Mathematics, University of Rome Tor Vergata, 
CUP E83C18000100006.  Andreas Leopold Knutsen acknowledges support from the Trond Mohn Foundation Project ``Pure Mathematics in Norway'' and grant 261756 of the
Research Council of Norway.

\section{Preliminaries}\label{sec:prel}

\subsection{}\label{ssec:uno}
In this paper we will denote by $\pi: X_n\rightarrow \PP^2$ the blow--up of the complex projective plane at  $n\geq 1$  very general points $p_1,\ldots, p_n$. We will denote by $E_1,\ldots, E_n$ the exceptional divisors  contracted by $\pi$ to  $p_1,\ldots, p_n$ respectively, and by $L$ the pull--back via $\pi$ of a general line of $\PP^2$. For integers $d, m_1,\ldots, m_n$, consider the linear system

\begin{equation}\label{eq:syst}
\Big |dL-\sum_{i=1}^nm_iE_i\Big |
\end{equation} 
on $X_n$.  If $d, m_1,\ldots, m_n$  are non--negative,  this is the strict transform on $X_n$ of the linear system of curves of degree $d$ in $\PP^2$ with multiplicities at least $m_1,\ldots, m_n$ at 
$p_1,\ldots, p_n$ respectively.  If one of the integers $m_1,\ldots, m_n$ is negative, e.g., $m_1$ is negative, then this means that $m_1E_1$ is in the fixed part of the linear system \eqref {eq:syst}, if this system is non--empty.
We will denote  the linear system  \eqref {eq:syst}, or the corresponding line bundle   $\O_{X_n}(dL-\sum_{i=1}^nm_iE_i)$,  by
$$(d;m_1,\ldots, m_n)$$
and we will use exponential notation for repeated multiplicities. 

We will set
$$
\xi_{n,m}=(m;1^n).
$$
This line bundle is ample as soon as $n\geq 3$ and $\xi_{n,m}^2=m^2-n>0$ (see \cite[Cor. p. 154]{Ku}) and it is very ample if $n\geq 3$ and 
$$
 m\geq 2\sqrt {n+4}-3
$$
(see \cite [Thm. 1] {ST}). This result is not optimal and a natural conjecture is that $\xi_{n,m}$ is very ample as soon as
$$
\frac {m(m+3)}2-n\geq 5
$$
(see again \cite [Conj. p. 2521] {ST}).   In any event, if $m_n$ is the minimum $m$ such that $\xi_{n,m}$ is very ample, we have
\begin{equation}\label{eq:mag}
m_n\leq \left \lceil{ 2\sqrt {n+4}-3}\right \rceil
\end{equation}
if $n\geq 3$, whereas $m_1=2$ and $m_2=3$.

We will need the following:

\begin{lemma}\label{lem:num} If $n\geq 3$, there is an $m< 2\sqrt n$ such that  $\xi_{n,m}$ is very ample.
\end{lemma}

\begin{proof} To prove the assertion one has to prove that $2\sqrt n> m_n$. By \eqref {eq:mag}, this is implied by $2\sqrt n> \left \lceil{ 2\sqrt {n+4}-3}\right \rceil$, which  holds if $2\sqrt n\geq 2\sqrt {n+4}-2$. This is true as soon as $n\geq 3$. \end{proof}

\subsection{} Given a linear system $\L=(d;m_1,\ldots, m_n)$,  its \emph{virtual dimension} is
$$
\dim_v(\L)=\frac {d(d+3)}2-\sum_{i=1}^n\frac {m_i(m_i+1)}2.
$$
If $\dim(\L)=\dim_v(\L)$ we will say that the system $	\L$ is \emph{regular}. 

One computes
$$
\dim_v(\L)=\chi(\L)-1
$$
hence, by the Riemann--Roch theorem, one has
$$
\dim(\L)=\dim_v(\L)+h^1(\L)-h^2(\L).
$$
Since  $h^2(\L)=h^0(K_n\otimes \L^\vee)$, where $K_n=(-3;-1^n)$ 
denotes the canonical bundle of $X_n$, and 
$$
K_n\otimes \L^\vee=(-3-d;-m_1-1,\ldots, -m_n-1)
$$
one has $h^2(\L)=0$ as soon as $d>-3$. In this case one has
\begin{equation}\label{eq:sequo}
\dim(\L)=\dim_v(\L)+h^1(\L)
\end{equation}
and $\L$ is regular if and only if $h^1(\L)=0$.

\begin{lemma}\label{lem:1} Let $\L$ be a linear system on $X_n$ of the form
$$
(d; 1^h,0^k, (-1)^{n-h-k})
$$
where $d\geq 0$. If $\dim_v(\L)\geq -1$, then  $\L$ is regular.
\end{lemma}

\begin{proof} Set $E=E_{h+k+1}+\cdots+E_n$.  For each curve $E_i\cong \PP^1$, for $h+k+1\leq i\leq n$ one has 
$\L\cdot E_i=-1$, hence $E$ is in the fixed part of $\L$, if this is non--empty. So we have
$\dim(\L)=\dim(\L(-E))$.
The linear system $ \L(-E)$ is $(d; 1^h,0^k)$. One has $\dim_v(\L)=\dim_v(\L(-E))$.
By $\dim_v(\L(-E))\geq -1$ and by the generality of the imposed simple base points $p_1,\ldots, p_h$, it follows that $\L(-E)$ is regular. Hence we have
$$
\dim(\L)=\dim(\L(-E))=\dim_v(\L(-E))=\dim_v(\L)
$$
as wanted.  \end{proof}

\subsection{} We will use the following result that is a consequence of  \cite[Cor. 4.6] {AC} (see \cite [Thm. 2.3]{CDGK} and \cite[Thms. 1.1, 1.3 and 1.4]{CC}): 

\begin{proposition}\label{thm:wdv} Consider a linear system of the form
$$
\L_{d,n,\delta,k}=(d;2^\delta,1^k, 0^{n-\delta-k})
$$
(and permuted multiplicities), with $k\geq 1$.
If
$$
\dim_v(\L_{d,n,\delta,k})=\frac {d(d+3)}2-3\delta-k\geq 0
$$
then $\L_{d,n,\delta,k}$ is regular, i.e., 
$$
h^1(X_n,\L_{d,n,\delta,k})=0
$$
and the general curve in $\L_{d,n,\delta,k}$ on $X_n$ is smooth and irreducible.
\end{proposition}

\subsection{}

We will need the following proposition (cf. also \cite[Prop. (2.1)]{Ca}):

\begin{proposition}\label{lem:ulr1} Let $S$ be a smooth irreducible projective surface and let $H$ be a very ample divisor on $S$.

A line bundle $\L \not \cong \O_S$ is Ulrich if and only if it is of the form $\L=\O_S(C)$, where $C$ is an effective divisor on $S$ satisfying \\
\begin{inparaenum}
\item [(i)] $C \cdot H = \frac{1}{2} H \cdot (3H+K_S)$;\\
\item  [(ii)] $\frac{1}{2} (C^2 - C \cdot K_S) + \chi(\O_S)-H^2=0$;\\
\item  [(iii)] $h^1(\O_C(K_S+H))=0$;\\
\item  [(iv)] the restriction map  $r: H^0(\O_S(K_S+2H)) \to H^0(\O_C(K_S+2H))$ 
is  injective or surjective.
\end{inparaenum}

Moreover, the divisor $C$ can be taken to be a smooth curve.
\end{proposition}

We will say that $\L$ as above is \emph{determined} by $C$.

\begin{proof}
  It is well-known that an Ulrich bundle is globally generated, whence it is of the form $\L=\O_S(C)$ for $C$ an effective nonzero divisor, which can even be taken to be a smooth curve. Hence, $\L$ is Ulrich if and only if
$h^i(\O_S(C-jH))=0$ for  $i=0,1,2, \; j=1,2$. 
  From the short exact sequence 
  \begin{equation} \label{eq:c3}
          0 \longrightarrow \O_S(-jH) \longrightarrow \O_S(C-jH)  \longrightarrow \omega_C(-K_S-jH) \longrightarrow 0
    \end{equation}
  we see that $\L$ is Ulrich if and only if
  \begin{eqnarray}
    \label{eq:c4}
    & h^0(\omega_C(-K_S-jH))=0, \; j=1,2, &\\
    \label{eq:c5}
    & \;\;\;\;\mbox{the coboundary maps} \;  H^1(\omega_C(-K_S-jH)) \to H^2(\O_S(-jH)) \\  \nonumber & \mbox{are isomorphisms}, \; j=1,2.&
  \end{eqnarray}
Clearly the vanishing for $j=2$ in \eqref{eq:c4} is implied by the one for $j=1$. Thus, by Serre duality \eqref{eq:c4}--\eqref{eq:c5} 
are  equivalent to
\begin{eqnarray}
    \label{eq:c6}
    & h^1(\O_C(K_S+H))=0, & \\ 
    \label{eq:c7}
  & \;\;\;\;\mbox{the restriction maps} \; r_j: H^0(\O_S(K_S+jH)) \to H^0(\O_C(K_S+jH)) \\
  & \nonumber \mbox{are isomorphisms}, \; j=1,2.&
  \end{eqnarray}
  By \eqref{eq:c6} and the fact that $h^1(\O_S(K_S+jH))=h^2(\O_S(K_S+jH))=0$, the domain and target of $r_j$ have dimensions $\chi(\O_S(K_S+jH))$ and $\chi(\O_C(K_S+jH))$, respectively. Morover, it is easy to see that $r_1$ is injective as soon as $r_2$ is. 
Hence, given \eqref{eq:c6}, condition \eqref{eq:c7} is equivalent to
  \begin{eqnarray}
    \label{eq:c8}
    & \chi(\O_S(K_S+jH))=\chi(\O_C(K_S+jH)), \; \; j=1,2, & \\
\label{eq:c9} & \;\;\;\;\mbox{the restriction map} \; r_2:H^0(\O_S(K_S+2H)) \to H^0(\O_C(K_S+2H)) \; \mbox{is injective}.&
\end{eqnarray}
Thus, $\O_S(C)$ is Ulrich if and only if \eqref{eq:c6}, \eqref{eq:c8} and \eqref{eq:c9} are satisfied.
Condition \eqref{eq:c6} is condition (iii) in the statement of the proposition, whereas \eqref{eq:c8} is equivalent to (i)--(ii) by Riemann--Roch. Finally,
\eqref{eq:c9} is equivalent to (iv), as the domain and target have the same dimensions, again by \eqref{eq:c8}.
\end{proof}

\section{Ulrich line bundles on $X_n$}\label{sec:linbund}

In this section we will prove the existence of Ulrich line bundles on $X_n$.  Note that the Del Pezzo case ($n\leq 6$ and $m=3$ in the notation below) has  already  been worked out in \cite [Prop. 2.19]{CKM} and \cite [Thm. 1.1]{PT}. 
This is our result:

\begin{thm}\label{thm:main1} Let   $n\geq 3$ be an integer  and let $m$ be an integer such that $\xi_{n,m}$ is very ample   on $X_n$ and $m\leq 2\sqrt {n}$  
(such an $m$ exists by Lemma \ref {lem:num}). 
Let $d$ be a positive integer such that
\begin{equation}\label{eq:prim}
\frac {2m-3-\sqrt {8n+1}}2\leq d\leq \frac {2m-3+\sqrt {8n+1}}2
\end{equation}
and 
\begin{equation}\label{eq:sec}
\frac {3(m-1)-\sqrt {4n-m^2+1}}2< d< \frac {3(m-1)+\sqrt {4n-m^2+1}}2.
\end{equation}
Set
$$
\delta=\frac {m^2}2-\frac m2 (2d+3)+\frac {d^2+3d+2}2=\frac {(d-m)(d-m+3)}2+1
$$
and
$$
k=n+3m(d+1)-\frac {m(5m-3)}2-(d^2+3d+2).
$$
Then $\delta$ and $k$ are integers such that
\begin{equation}\label{eq:delta}
0\leq \delta\leq n,
\end{equation}
\begin{equation}\label{eq:k}
1\leq k\leq n,
\end{equation}
and
\begin{equation}\label{eq:kd}
\delta+k\leq n.
\end{equation}

Moreover 
$$
\L_{d,n,k,\delta}=(d;2^\delta, 1^k, 0^{n-\delta-k})
$$
(and permuted multiplicities) is an Ulrich line bundle  on $X_n$ with respect to $\xi_{n,m}$. 
\end{thm}

\begin{proof} Take for granted \eqref {eq:delta}, \eqref {eq:k} and \eqref {eq:kd} for the time being. A direct computation of 
$\dim_v(\L_{d,n,k,\delta})$ as in Proposition \ref{thm:wdv} shows that
$$
\dim_v(\L_{d,n,k,\delta})=m^2-n-1\geq 0
$$
because
$$
m^2-n=\xi_{n,m}^2>0.
$$
By Proposition \ref {thm:wdv}, the system $\L_{d,n,k,\delta}$ is non--empty and the general curve $C$ in $\L_{d,n,k,\delta}$ is smooth  since $k \geqslant 1$ by \eqref{eq:k}. Now we apply Proposition \ref {lem:ulr1} to $S=X_n$, $H=\xi_{n,m}$ and $C$ as above. To prove the theorem we have to check conditions (i)--(iv) in the statement of  Proposition  
\ref{lem:ulr1}. 

As for (i), with an easy computation we have
$$
C\cdot \xi_{n,m}=dm-2\delta-k=\frac {3m(m-1)}2-n=\frac 12 \xi_{n,m}\cdot(3\xi_{n,m}  +  K_n)
$$
as wanted.

As for (ii), note that $\chi(\O_{X_n})=1$, so
$$
\frac{1}{2} (C^2 - C \cdot K_n) + \chi(\O_S)-\xi_{n,m}^2=\frac 12(d^2+3d-6\delta-2k )-m^2+n+1
$$
which is easily computed to be 0, as desired.

As for (iii), suppose, by contradiction, that $h^1(\O_{C}(K_n+\xi_{n,m}))>0$. So the divisors cut out by the system $K_n+\xi_{n,m}$ on $C$ are special, i.e., they are contained in divisors of the canonical series, cut out on $C$ by the system $K_n+\L_{d,n,k,\delta}$. 
Note that 
$$
K_n+\xi_{n,m}=(m-3;0^n) \quad \text{and}\quad K_n+\L_{d,n,k,\delta}=(d-3;1^\delta, 0^k, (-1)^{n-\delta-k}).
$$
Let $D$ be a general curve in the system $K_n+\xi_{n,m}$ (i.e., a general curve of degree $m-3$), that cuts on $C$ a divisor consisting of $d(m-3)$ distinct points. This divisor has to be contained in a curve $D'$ of  the system $K_n+\L_{d,n,k,\delta}$. Since $D$ has degree $m-3$ and $D'$ has degree $d-3$ and  $d(m-3)>(d-3)(m-3)$, 
 by B\'ezout's theorem $D$ is contained in $D'$. This implies that the system $K_n+\L_{d,n,k,\delta}-(K_n+\xi_{n,m})=\L_{d,n,k,\delta}-\xi_{n,m}$ is effective. We have
$$
\L_{d,n,k,\delta}-\xi_{n,m}=(d-m;1^\delta, 0^k,(-1)^{n-\delta-k})
$$
 whose virtual dimension is
$$
\frac {(d-m)(d-m+3)}2-\delta=-1.
$$
Then, by Lemma \ref {lem:1}, $\L_{d,n,k,\delta}-\xi_{n,m}$ is empty, a contradiction. 

Finally, as for (iv) we want to prove that the map 
$$
r: H^0(K_n+2\xi_{n,m})\to H^0((K_n+2\xi_{n,m})_{|C})
$$
is injective. Note that 
$$K_n+2\xi_{n,m}=(2m-3;1^n).$$
If $2m-3<d$ it is clear that $r$ is injective. If  $2m-3\geq d$, the kernel of $r$ is $H^0( K_n + 2\xi_{n,m}-\L_{d,n,k,\delta})$. Now
$$
 K_n +  2\xi_{n,m}-\L_{d,n,k,\delta}=(2m-3-  d ;(-1)^\delta, 0^k, 1^{n-\delta-k})
$$
 whose virtual dimension is 
$$
\frac {(2m-3-d)(2m-d)}2-n+  k + \delta  =-1.
$$
By Lemma \ref {lem:1} we have that $ K_n + 2\xi_{n,m}-\L_{d,n,k,\delta}$ is empty, i.e., $h^0( K_n + 2\xi_{n,m}-\L_{d,n,k,\delta})=0$ and $r$ is injective, as wanted. 

To finish the proof, we are left to prove \eqref {eq:delta}, \eqref {eq:k} and \eqref {eq:kd}. This is a mere computation. For instance, $\delta\geq 0$ is equivalent to
$$
d^2-(2m-3)d+m^2-3m+2\geq 0.
$$
The discriminant of this quadratic in $d$ is 
$$
(2m-3)^2-4(m^2-3m+2)=1,
$$
hence we need
$$
\text{either}\quad d\geq \frac {2m-3+1}2=m-1, \quad \text {or}\quad d\leq \frac {2m-3-1}2=m-2
$$
which is always true.

The inequality $\delta\leq n$ is equivalent to
\begin{equation}\label{eq:primo}
d^2-(2m-3)d+m^2-3m+2-2n\leq 0.
\end{equation}
The discriminant in $d$ is computed to be $8n+1$. Since  \eqref {eq:prim} holds, \eqref {eq:primo} also holds.

The inequality $k> 0$ is equivalent to
\begin{equation}\label{eq:primos}
2d^2-6d(m-1)+5m^2-9m+4-2n< 0.
\end{equation}
The discriminant in $d$ is computed to be $4n-m^2+1$ which is non--negative  by assumption. Since \eqref {eq:sec} holds, then  \eqref {eq:primos} holds as well.

The inequality $k\leq n$ is equivalent to 
$$
2d^2-6d(m-1)+5m^2-9m+4\geq  0.
$$
The discriminant is computed to be $4(1-m^2)<0$, hence $k\leq n$ holds. 

Finally $\delta+k\leq n$ is equivalent to
$$
d^2-(4m-3)d+4m^2-6m+  2  \geq 0
$$
whose discriminant is 1, so that $\delta+k\leq n$ if
$$
\text{either}\quad d\geq \frac {4m-3+1}2=2m-1, \quad \text {or}\quad d\leq \frac {4m-3-1}2=2m-2
$$
which is always true.
\end{proof}

As an immediate consequence we have:

\begin{corollary}\label{cor:inf} There is a sequence $\{h_n\}_{n\in \mathbb N}$ such that $\lim_{n\to +\infty}h_n=+\infty$ and such that for $n\gg 0$ there are $h_n$ distinct Ulrich line bundles on $X_n$ with respect to $\xi_{n,m}$.
\end{corollary}

\begin{remark}\label{rem:better} The hypotheses of Theorem \ref {thm:main1} could be a bit relaxed. The strict inequalities in \eqref {eq:sec} are required in order that $k\geq 1$, and this is used to imply that the general curve $C$ in $\L_{d,n,\delta,k}$ is smooth. However $k\geq 1$ is not necessary for the general curve $C$ in $\L_{d,n,\delta,k}$ to be smooth. For example, for $n=2$ and $m=3$, the general curve in  $\L_{3,2,1,0}=(3;2,0)$ is smooth and this line bundle is Ulrich (see also \cite [Prop. 2.19]{CKM} and \cite [Thm. 1.1]{PT}). We will use this later. 
\end{remark}

\section{Classification of Ulrich line bundles on $X_n$ with $7\leq n\leq 10$}\label{sec:class}

Theorem \ref {thm:main1} proves the existence of several Ulrich line bundles on $X_n$ but does not give a full classification of such line bundles. However, in principle, it is possible to 
pursue such a classification. We want to show how this works   for Ulrich line bundles with respect to $\xi_{n,4}$, in the case in which $m=4$ is the minimum such that $\xi_{n,m}$ is very ample. This   corresponds to $7\leq n\leq 10$. As we said already, the Del Pezzo case $m=3$  is well known, hence we do  not dwell on it here.

As we know from Proposition \ref {lem:ulr1}, an Ulrich line bundle on $X_n$ is of the form $\O_{X_n}(C)$ where $C$ is a smooth curve on $X_n$ verifiying properties (i)--(iv). Let us restate properties (i)--(iv) in our setting, with $7\leq n\leq 10$, in which $m=4$. We have $H=\xi_{n,4}$ and $C$ belongs to a system $(d;m_1,\ldots, m_n)$, and it is easy to check that properties (i)--(iv) read now as follows: \\
\begin{inparaenum}
\item [(i)] $\xi_{n,4}\cdot C=18-n$;\\
\item [(ii)] $p_a(C)=d-2$;\\
\item [(iii)] $h^1(\O_C(L))=0$ (recall that $L$ is the pull--back to $X_n$ of a general line of $\PP^2$);\\
\item  [(iv)] the restriction map  $r: H^0(\O_S(K_n+2\xi_{n,4})) \to H^0(\O_C(K_n+2\xi_{n,4}))$ is either injective or surjective, where
$K_n+2\xi_{n,4}=(5;1^n)$.
\end{inparaenum}

We need a preliminary lemma:

\begin{lemma}\label{lem:red} In the above setting, the curve $C$ is irreducible.
\end{lemma}

\begin{proof} First we notice that the case $d=1$ cannot happen. In this case in fact, since $\O_{X_n}(C)$ is globally generated, $C$ would be clearly irreducible while from (ii) we would have $p_a(C)=-1$ a contradiction. So we have $d\geq 2$ and therefore $p_a(C)=d-2\geq 0$.

If $C$ were reducible, since $\O_{X_n}(C)$ is globally generated, this means that $|C|$ would be composed with a pencil $|F|$, i.e., $C\sim aF$ with $a\geq 2$. Then $p_a(C)=ap_a(F)-a+1\geq 0$ hence $p_a(F)\geq 1$.  Consider the exact sequence
$$
0\longrightarrow \O_{X_n}((a-1)F)\longrightarrow  \O_{X_n}(aF)= \O_{X_n}(C)\longrightarrow \O_{F}(C)=\O_F\longrightarrow 0
$$
and the cohomology sequence
$$
H^1(\O_{X_n}(C))\longrightarrow H^1(\O_F)\longrightarrow H^2(\O_{X_n}((a-1)F)).
$$
We would have $h^1(\O_F)=p_a(F)\geq 1$ and $h^2(\O_{X_n}((a-1)F))=0$, hence $h^1(\O_{X_n}(C))>0$. This gives a contradiction since  $h^1(\O_{X_n}(C))=0$ (see, for instance, \cite [(3.1)]{be2}).\end{proof}

\begin{thm}\label{thm:class} The Ulrich line bundles on $X_n$ with $7\leq n\leq 10$ with respect to $\xi_{n,4}$ are:\\
\begin{inparaenum}
\item [(a)] $(6;2^6, 1^{n-6})$;\\
\item [(b)] $(5;2^3, 1^{n-4},0)$;\\
\item [(c)] $(4;2, 1^{n-4},0^3)$;\\
\item [(d)] $(3;1^{n-6},0^6)$;\\
\item [(e)] $(2;0^{10})$ only for $n=10$;\\
\item [(f)] $(7;2^{10})$ only for $n=10$,\\
\end{inparaenum}
with permutations of the multiplicities.
\end{thm}

\begin{proof} Let $\O_{X_n}(C)$ be an Ulrich line bundle on $X_n$, with $C$ smooth and irreducible, by Lemma \ref {lem:red}. We suppose $\O_{X_n}(C)$ is of the form
$(d;m_1,\ldots, m_n)$ on $X_n$.

Consider the embedding $\phi_{\xi_{n,4}}: X_n\to S\subset \PP^{14-n}$, where $S$ is the image of $X_n$, and denote still by $C$ the image of $C$ on $S$. If $C$ is degenerate, then
clearly $d\leq 4$. Otherwise $C$ is a non--degenerate smooth irreducible curve of degree $18-n$ in $\PP^{14-n}$ (see property (i) above). By Castelnuovo's bound, we have $p_a(C)\leq 4$ if $7\leq n\leq 9$ and $g(C)\leq 5$ if $n=10$. In addition, if $n=10$ and $g(C)=5$, then $C\subset \PP^4$ is a canonical curve, i.e., $\xi_{10,4}$ cuts out on $C$ the canonical series, hence we  are in case (f).  By property (ii), we have $d\leq 6$ if $7\leq n\leq 9$ and $d\leq 7$ if $n=10$. On the other hand we have $d\geq 2$ (see the 
proof of Lemma \ref {lem:red}). 

Suppose $d=6$, hence $p_a(C)=4$. Then we have
\begin{align*}
& 10-\sum_{i=1}^n \frac {m_i(m_i-1)}2=4\\
&24-\sum_{i=1}^n m_i=18-n
\end{align*}
i.e., 
\begin{align*}
& 6-\sum_{i=1}^n \frac {m_i(m_i-1)}2=0\\
& 6-\sum_{i=1}^n (m_i-1)=0
\end{align*}
whence
$$
\sum_{i=1}^n \frac {m_i(m_i-1)}2=\sum_{i=1}^n (m_i-1)
$$
i.e.,
$$
\sum_{i=1}^n \frac {(m_i-1)(m_i-2)}2=0
$$
which yields $m_i\leq 2$ for  $ 1  \leq i\leq n$. This implies that we are in case (a). A similar computation shows that in case $d=5$ the only possible cases are the one in (b) and $(5;3, 1^{n-1})$. This line bundle however is not Ulrich because property (iv) is not verified, since the kernel of the map $r$ is $H^0(\L)$ with $\L = 
K_n + 2 \xi_{n,m} - C =(0;(-2), 0^{n-1})$, that is non--zero. 
If $2\leq d\leq 4$ it is easily seen that the remaining possible cases are the ones in (c)--(e). 

Finally one has to check that in all cases (a)--(f) we do have Ulrich line bundles. To see this, by Proposition   \ref {lem:ulr1} we have only to prove that properties (iii) and (iv) above are verified. Property (iii) is trivially verified by degree reasons in cases (b)--(e). In case (a) it is verified because otherwise the 
system $(2;1^6, 0^{n-6})$ would be non--empty, a contradiction. In case (f) it is verified because otherwise the 
system $(3;1^{10})$ would be non--empty, a contradiction again.
As for property (iv), note that the map $r$ is clearly injective in cases (a) and (f). As for the other cases one looks at the kernel of the map $r$, which is the $H^0$ of the following line bundles:\\
\begin{inparaenum}
\item [$\bullet$] $(0;(-1)^3, 0^{n-4}, 1)$ in case (b);\\
\item [$\bullet$] $(1;-1, 0^{n-4}, 1^3)$ in case (c);\\
\item [$\bullet$] $(2; 0^{n-6}, 1^6)$ in case (d);\\
\item [$\bullet$] $(3;1^{10})$ in case (e),\\
\end{inparaenum} 
and in all these cases this $H^0$ is zero, so $r$ is injective. \end{proof}

\section{Higher rank Ulrich vector bundles on $X_n$}\label{sec:high}

In this section we will construct higher rank  slope--stable  Ulrich vector bundles on $X_n$  and we will compute the dimensions of the moduli spaces of the constructed bundles. 

In the whole section  $n\geq 2$ will be an integer and $m$ will be an integer such that $\xi_{n,m}$ is very ample with $m=3$ if $n=2$ and 
$m< 2\sqrt {n}$ if $n\geq 3$ (cf. Lemma \ref {lem:num}).

 We start by defining 
\begin{eqnarray*}
L_1 & = & \begin{cases} \Big (\frac {3(m-1)}2; 2^{\frac {m^2-1}8},1^{n-\frac {m^2-1}4}, 0^{\frac {m^2-1}8}\Big),  & \text{if $m$ is odd,}\\
 \Big ( \frac 32m-1; 2^{\frac {m(m+2)}8},1^{n-\frac {m^2}4}, 0^{\frac {m(m-2)}8}\Big),  & \text{if $m$ is even,} \\
\end{cases} \\
L_0{} & = & \begin{cases} \Big (\frac {3(m-1)}2; 0^{\frac {m^2-1}8},1^{n-\frac {m^2-1}4}, 2^{\frac {m^2-1}8}\Big), & \text{if $m$ is odd,}\\
 \Big ( \frac 32m-1; 0^{\frac {m(m-2)}8}, 1^{n-\frac {m^2}4}, 2^{\frac {m(m+2)}8}\Big),  & \text{if $m$ is even.} \\
\end{cases}
\end{eqnarray*}

\begin{lemma} \label{lemma:0}
  The line bundles $L_0$ and $L_1$ are Ulrich  with respect to $\xi_{n,m}$  and satisfy
\begin{equation}\label{eq:cap}
 h^i(L_0-L_1)=h^i(L_1-L_0)=0, \; \; \mbox{for} \; \; i=0,2, 
 \end{equation}
  and
\begin{equation}\label{eq:capo} h^1(L_0-L_1)=h^1(L_1-L_0)=   \begin{cases}
      \frac {m^2-3}2, & \text{if $m$ is odd}\\
 \frac {m^2-m-2}2, & \text{if $m$ is even}. 
    \end{cases}
     \end{equation}
\end{lemma}

\begin{proof}
  Let $\delta$ be as in Theorem \ref{thm:main1}.  Note  that
\begin{equation}\label{eq:panza}
n-\delta-k=\delta+\frac {3m(m-1)}2-md.
\end{equation}

We first treat the case where $m$ is odd.

Take $d=\frac {3(m-1)}2$  and assume first $n\geq 3$.  Such a value of $d$ is clearly compatible with \eqref {eq:sec}. It is also compatible with \eqref {eq:prim}. Indeed
$$
3(m-1)-(2m-3)+\sqrt {8n+1}=m+\sqrt {8n+1}>0
$$
and 
$$
2m-3+\sqrt {8n+1}-3(m-1)=\sqrt {8n+1}-m>0
$$
because   $m< 2\sqrt {n}$.  For such a value of $d$, we get 
$\delta=n-\delta-k$ from \eqref {eq:panza}. More precisely, an easy computation shows that for such a $d$ one has
$$
\delta= \frac {m^2-1}8, \quad \text{hence} \quad k=n-\frac {m^2-1}4.
$$ 
Hence we get Ulrich line bundles on $X_n$ of the form
$$
\L_{\frac {3(m-1)}2,n,\frac {m^2-1}8,n-\frac {m^2-1}4}=\Big (\frac {3(m-1)}2, 2^{\frac {m^2-1}8},1^{n-\frac {m^2-1}4}, 0^{\frac {m^2-1}8}\Big)
$$
and permutations of the multiplicities. In particular,  we see that $L_0$ and $L_1$ defined above are Ulrich.

  Up to permutations of multiplicities,  one has
$$
L_1-L_0=  L_0-L_1 =  \Big (0; 2^{\frac {m^2-1}8},0^{n-\frac {m^2-1}4}, (-2)^{\frac {m^2-1}8}\Big).
$$
 Hence we see that $h^0(L_1-L_0)=  h^0(L_0-L_1) =  0$ and, by Serre duality, $h^2(L_1-L_0)=h^0(K_n+L_0-L_1)=0$. 
By Riemann--Roch and an easy computation,
\[ h^1(L_1-L_0)=-\chi(L_1-L_0)=\frac {m^2-3}2.\]
The cohomology of $L_0-L_1$ is computed in the same way.

 In the case $n=2$, $m=3$, the line bundles $L_0$ and $L_1$ are Ulrich (see Remark \ref {rem:better}) and \eqref {eq:cap} and \eqref {eq:capo} are verified as before.

We next treat the case where $m$ is even.

Take  $d=\frac 32 m-1$.  As in the odd case, such a value of $d$ is easily seen to be compatible with \eqref {eq:prim} and with \eqref {eq:sec}  (for verifying \eqref {eq:sec} one needs $m<2\sqrt n$).  For this value of $d$ one easily computes
$$
\delta= \frac {m(m+2)}8,  \quad k=n-\frac {m^2}4 \quad \text{and}\quad n-\delta-k=\frac {m(m-2)}8.
$$
Hence we get Ulrich line bundles of the form
$$
\L_{\frac 32m-1,n, \frac {m(m+2)}8, n-\frac {m^2}4}=\Big ( \frac 32m-1, 2^{\frac {m(m+2)}8},1^{n-\frac {m^2}4}, 0^{\frac {m(m-2)}8}\Big)
$$
and permutations of the multiplicities.
In particular,  $L_0$ and $L_1$ defined above are Ulrich.

One has
$$
L_1-L_0=\Big (0; 2^{\frac {m(m-2)}8},1^{\frac m2},
0^{n-\frac {m(m+2)}4},(-1)^{\frac m2},
(-2)^{\frac {m(m-2)}8}\Big).
$$
 Again we see that $h^0(L_1-L_0)= h^2(L_1-L_0)=0$. The computation of $h^1(L_1-L_0)$ then again follows by Riemann-Roch. The cohomology of $L_0-L_1$ is computed in the same way.
\end{proof}

For simplicity we will set 
$$
h:=h^1(L_0-L_1)=h^1(L_1-L_0)=
\begin{cases} \frac {m^2-3}2 \quad \text{if $m$ is odd}\\
 \frac {m^2-m-2}2, \quad \text{if $m$ is even}. \\
\end{cases}
$$

 As  $n \geq 2$ then  $m \geq 3$, whence $h \geq 3$.

To construct higher rank Ulrich bundles on $X_n$ we proceed as follows. 

Since $\Ext^1(L_0,L_1)\cong H^1(L_1-L_0) \cong \CC^h$, we have a non--split extension
\begin{equation}
    \label{eq:1r1}
\xymatrix{ 0 \ar[r] & \E_1:=L_1 \ar[r] & \E_2 \ar[r] & L_0 \ar[r] & 0,}
  \end{equation}
  where $\E_2$ is a rank--two vector bundle, necessarily Ulrich, as $L_0$ and $L_1$ are. We proceed taking extensions
 \[
\xymatrix{ 0 \ar[r] & \E_2 \ar[r] & \E_3 \ar[r] & L_1 \ar[r] & 0,}
  \] 
\[
\xymatrix{ 0 \ar[r] & \E_3 \ar[r] & \E_4 \ar[r] & L_0 \ar[r] & 0,}
  \]
  and so on; that is, defining
  \begin{equation}
    \label{eq:jr}
    \epsilon_r=
    \begin{cases}
      0, & \mbox{if $r$ is even}, \\
      1, & \mbox{if $r$ is odd},
    \end{cases}
  \end{equation}
  we take successive extensions $[\E_{r+1}] \in \Ext^1(L_{\epsilon_{r+1}},\E_r)$ for all $r \geq 1$:
  \begin{equation}
    \label{eq:1}
    \xymatrix{ 0 \ar[r] & \E_r \ar[r] & \E_{r+1} \ar[r] & L_{\epsilon_{r+1}} \ar[r] & 0.}
 \end{equation}
  A priori we do not know that we can always take {\it non--split} such extensions; this we will prove in a moment. In any case, all $\E_{r}$ are Ulrich vector bundles of rank $r$, as extensions of Ulrich bundles are again Ulrich.

  \begin{lemma} \label{lemma:1}
    Let $L$ be $L_0$ or $L_1$. Then, for all $r \geq 1$ we have
    \begin{itemize}
    \item[(i)] $h^2(\E_r \* L^*)=0$,
      \item[(ii)] $h^2(\E_r^* \* L)=0$,
      \item[(iii)] $h^1(\E_r \* L_{\epsilon_{r+1}}^*)\geq h \geq 3$.
      \end{itemize}
  \end{lemma}
  
  \begin{proof}
    We prove (i)-(ii) by induction on $r$. Regarding (i), it holds for
    $r=1$ since $\E_1=L_1$, by Lemma \ref{lemma:0}. Assuming it holds for $r$, we have, by tensoring \eqref{eq:1} by $L^*$, that
    \[ h^2(\E_{r+1} \* L^*) \leq h^2(\E_r \* L^*)+h^2(L_{\epsilon_{r+1}} \* L^*)=0,\]
    by the induction hypothesis and Lemma \ref{lemma:0}, since $L_{\epsilon_{r+1}} \* L^*$ equals one of $\O_X$, $L_0-L_1$ or $L_1-L_0$.

    A similar reasoning, tensoring the dual of \eqref{eq:1} by $L$, proves (ii).

    To prove (iii), first note that it holds for $r=1$, as $h^1(\E_1 \* L_{\epsilon_{2}}^*)=h^1(L_1-L_0)= h$ by Lemma \ref{lemma:0}. Then, for any $r \geq 1$,
tensor \eqref{eq:1} by $L_{\epsilon_{r+2}}^*$. Using that $h^2(\E_r\* L_{\epsilon_{r+2}}^*)=0$ by (i), we see that 
$h^1(\E_{r+1} \* L_{\epsilon_{r+2}}^*)\geq h^1(L_{\epsilon_{r+1}}\* L_{\epsilon_{r+2}}^*)=h  \geqslant 3 $, by Lemma \ref{lemma:0}.
 \end{proof}

By (iii) of the last lemma, we have that $\dim (\Ext^1(L_{\epsilon_{r+1}},\E_r))=h^1(\E_r \* L_{\epsilon_{r+1}}^*)>0$ for all $r \geq 1$, which means that we can always pick {\it non--split} extensions of the form \eqref{eq:1}. We will henceforth do so.

  \begin{lemma} \label{lemma:2}
    For all $r \geq 1$ we have
    \begin{itemize}
    \item[(i)] $h^1(\E_{r+1} \* L_{\epsilon_{r+1}}^*)=h^1(\E_r \* L_{\epsilon_{r+1}}^*)-1$,
    \item[(ii)] $h^1(\E_r \* L_{\epsilon_{r+1}}^*)=\lfloor \frac{r+1}{2}\rfloor(h-1) +1$,

\item[(iii)] $h^2(\E_r \* \E_r^*)=0$,
\item[(iv)] $\chi(\E_r \* L_{\epsilon_{r+1}}^*)=-\left\lfloor\frac{r+1}{2}\right\rfloor(h-1)-\epsilon_r$,
\item[(v)] $\chi(L_{\epsilon_{r}} \* \E_r^* )=-\left\lfloor\frac{r+1}{2}\right\rfloor(h-1)+\epsilon_rh$,
\item[(vi)] $\chi(\E_r \* \E_r^*)=-\frac{1}{2}\left(r^2-\epsilon_r\right)\left(h-1\right)+\epsilon_r$.
  \item[(vii)]  the slope of $\E_r$ is   $\mu(\E_r)=L_0\cdot \xi_{m,n}=L_1\cdot \xi_{m,n}$.
    \end{itemize}
\end{lemma}

  \begin{proof}
    (i) Since $\Ext^1(L_{\epsilon_{r+1}},\E_r)\cong H^1(\E_r \* L_{\epsilon_{r+1}}^*)$ and the sequence \eqref{eq:1} is constructed by taking a non--zero element therein, the coboundary map $H^0(\O_X) \to H^1(\E_r \* L_{\epsilon_{r+1}}^*)$
of \eqref{eq:1} tensored by $L_{\epsilon_{r+1}}^*$,  i.e., 
\begin{equation}
    \label{eq:dag}
    \xymatrix{ 0 \ar[r] & \E_r\*L_{\epsilon_{r+1}}^* \ar[r] & \E_{r+1}\*L_{\epsilon_{r+1}}^* \ar[r] & \O_{X_n} \ar[r] & 0,}
 \end{equation}
 is non--zero. Thus, (i) follows from the cohomology of \eqref{eq:dag}.
 
(ii) We use induction on $r$. 
    For $r=1$, the right hand side of the formula yields $h$, whereas the left hand side equals $h^1(\E_1 \* L_0^*)=h^1(L_1-L_0)$; thus the formula is correct by Lemma \ref{lemma:0}.

    Assume now that the formula holds for $r$. Tensoring \eqref{eq:1} by $L_{\epsilon_{r+2}}^*$, we obtain
\begin{equation}
    \label{eq:dagdag}
    \xymatrix{ 0 \ar[r] & \E_r\* L_{\epsilon_{r+2}}^* \ar[r] & \E_{r+1}\* L_{\epsilon_{r+2}}^* \ar[r] & L_{\epsilon_{r+1}}\* L_{\epsilon_{r+2}}^* \ar[r] & 0.}
 \end{equation}
 We have $h^0(L_{\epsilon_{r+1}}\* L_{\epsilon_{r+2}}^*)=0$ and  $h^1(L_{\epsilon_{r+1}}\* L_{\epsilon_{r+2}}^*)=h$ by Lemma \ref{lemma:0}, and $h^2(\E_r\* L_{\epsilon_{r+2}}^*)=0$ by Lemma \ref{lemma:1}. Thus,
 \[ h^1(\E_{r+1} \* L_{\epsilon_{r+2}}^*)=h+h^1(\E_r\* L_{\epsilon_{r+2}}^*)=h+h^1(\E_r\* L_{\epsilon_{r}}^*)
 \]
 Using (i) and the induction hypothesis, this equals
 \[ h+\left(h^1(\E_{r-1} \* L_{\epsilon_{r}}^*)-1\right)= h+\left(\left\lfloor \frac{(r-1)+1}{2}\right\rfloor(h-1) +1\right)-1 = \left \lfloor \frac{(r+1)+1}{2}\right\rfloor(h-1) +1,\]
 showing that the formula holds for $r+1$.

(iii)  We again use induction on $r$. For $r=1$ (iii) says that $h^2(L_1 -L_1)=h^2(\O_{X_n})=0$, which is correct. Assume now that (iii) holds for $r$. The cohomology of \eqref{eq:1} tensored by $\E_{r+1}^*$ and Lemma \ref{lemma:1}(ii) yield
\begin{equation} \label{eq:zaniolo}
 h^2( \E_{r+1} \* \E_{r+1}^*) \leq h^2( \E_{r} \* \E_{r+1}^*) + h^2( L_{\epsilon_{r+1}} \* \E_{r+1}^*) =h^2( \E_{r} \* \E_{r+1}^*).
\end{equation}
The cohomology of the dual of \eqref{eq:1} tensored by $\E_{r}$ and Lemma \ref{lemma:1}(i) yield
\begin{equation} \label{eq:pogba}
h^2( \E_{r} \* \E_{r+1}^*) \leq h^2(\E_r \* L_{\epsilon_{r+1}}^*)+h^2(\E_{r} \* \E_{r}^*)= h^2(\E_{r} \* \E_{r}^*).
\end{equation}
Now \eqref{eq:zaniolo}--\eqref{eq:pogba} and the induction hypothesis yield
$h^2( \E_{r+1} \* \E_{r+1}^*)=0$, as desired.

(iv)  For $r=1$ (iv) reads 
$\chi(L_1 -L_0)=-h$, which is correct by Lemma \ref{lemma:0}. For $r=2$ (iv) reads  $\chi(\E_2 \* L_1^*)=-h+1$; using sequence \eqref{eq:dagdag} with $r=1$, one computes $\chi(\E_2 \* L_1^*)=\chi(\O_{X_n})+\chi(L_0-L_1)=1-h$, by Lemma \ref{lemma:0}, so again the formula is correct.

Assume now that the formula holds up to a certain $r\geq 2$.  From  \eqref{eq:dagdag} and Lemma \ref{lemma:0} we find
\begin{equation} \label{eq:maz1}
  \chi(\E_{r+1}\* L_{\epsilon_{r+2}}^*) = \chi(\E_r\* L_{\epsilon_{r+2}}^*)+ \chi(L_{\epsilon_{r+1}}\* L_{\epsilon_{r+2}}^*)= \chi(\E_r\* L_{\epsilon_{r}}^*)-h.
  \end{equation}
   Then  \eqref{eq:dag} (with $r$ replaced by $r-1$) yields
  \begin{equation}
    \label{eq:maz2}
    \chi(\E_r\* L_{\epsilon_{r}}^*)=\chi(\E_{r-1}\* L_{\epsilon_{r}}^*)+\chi(\O_{X_n})= \chi(\E_{r-1}\* L_{\epsilon_{r}}^*)+1.
  \end{equation}
  Inserting into \eqref{eq:maz1} and using the induction hypothesis, we get
  \begin{eqnarray*}
    \chi(\E_{r+1}\* L_{\epsilon_{r+2}}^*) &=& \chi(\E_{r-1}\* L_{\epsilon_{r}}^*)+1-h \\
    & = & -(h-1)\left\lfloor\frac{(r-1)+1}{2}\right\rfloor-\epsilon_{r-1}+1-h \\
                                          & = & -(h-1)\left(\left\lfloor\frac{r}{2}\right\rfloor+1\right) -\epsilon_{r-1}\\
 & = & -(h-1)\left\lfloor\frac{(r+1)+1}{2}\right\rfloor -\epsilon_{r+1},   
\end{eqnarray*}
proving that the formula holds also for $r+1$.

(v)  For $r=1$ (v) reads 
$\chi(L_1 -L_1)=1$, which is correct. For $r=2$ (v) reads  $\chi(L_0\* \E_2^*)=-h+1$; from the dual of sequence \eqref{eq:1r1} tensored by $L_0$, one computes $\chi(L_0\* \E_2^*)=\chi(\O_{X_n})+\chi(L_0-L_1)=1-h$, by Lemma \ref{lemma:0}, so again the formula is correct.

Assume now that the formula holds up to a certain $r\geq 2$.  From  the dual of sequence \eqref{eq:1} tensored by $L_{\epsilon_{r+1}}$
 we find
\begin{eqnarray} \label{eq:maz3}
  \chi(L_{\epsilon_{r+1}} \* \E_{r+1}^* ) & = &   \chi(L_{\epsilon_{r+1}} \* L_{\epsilon_{r+1}}^*)+ \chi(L_{\epsilon_{r+1}}\* \E_{r}^*) \\
  \nonumber & = &  \chi(\O_{X_n})+\chi(L_{\epsilon_{r+1}}\* \E_{r}^*) 
 =  1+\chi(L_{\epsilon_{r+1}}\* \E_{r}^*).
  \end{eqnarray}
   The  dual of sequence \eqref{eq:1} with $r$ replaced by $r-1$ tensored by $L_{\epsilon_{r+1}}$, together with Lemma \ref{lemma:0}, yields
  \begin{equation}
    \label{eq:maz4}
\chi(L_{\epsilon_{r+1}}\* \E_{r}^*)=\chi(L_{\epsilon_{r+1}} \* L_{\epsilon_{r}}^*)+ \chi(L_{\epsilon_{r+1}}\* \E_{r-1}^*)= -h+\chi(L_{\epsilon_{r-1}}\* \E_{r-1}^*).
\end{equation}
  Inserting into \eqref{eq:maz3} and using the induction hypothesis, we get
  \begin{eqnarray*}
    \chi(L_{\epsilon_{r+1}} \* \E_{r+1}^* ) &=& 1-h+\chi(L_{\epsilon_{r-1}}\* \E_{r-1}^*) \\
                                            & = &
1-h-(h-1)\left\lfloor\frac{(r-1)+1}{2}\right\rfloor+\epsilon_{r-1}h \\
                                          & = & -(h-1)\left(\left\lfloor\frac{r}{2}\right\rfloor+1\right) +\epsilon_{r-1}h\\
 & = & -(h-1)\left\lfloor\frac{(r+1)+1}{2}\right\rfloor +\epsilon_{r+1}h,   
\end{eqnarray*}
proving that the formula holds also for $r+1$.

(vi) We check the given formula for $r=1,2$.

We have $\chi(\E_1 \* \E_1^*)=\chi(L_1-L_1)=\chi(\O_X)=1$, which fits with the given formula for $r=1$. 

From \eqref{eq:1r1} tensored by $\E_2^*$ we get
\begin{equation}
  \label{eq:cr1}
  \chi(\E_2 \* \E_2^*)=\chi(L_1 \* \E_2^*)+\chi(L_0 \* \E_2^*)\stackrel{(v)}{=} \chi(L_1 \* \E_2^*)-(h-1).
\end{equation}
From the dual of  \eqref{eq:1r1}  tensored by $L_1$ and Lemma \ref{lemma:0} we get
\begin{equation}
  \label{eq:cr2}
  \chi(L_1 \* \E_2^*)=\chi(L_1 -L_1)+\chi(L_1 -L_0)=\chi(\O_{X_n})-h=1-h.
\end{equation}
Combining \eqref{eq:cr1} and \eqref{eq:cr2}, we get $\chi(\E_2 \* \E_2^*)=-2(h-1)$, which again fits with the given formula for $r=2$.

Assume now that the given formula is valid up to a certain $r$. From \eqref{eq:1} tensored by $\E_{r+1}^*$ and successively the dual of  \eqref{eq:1}  tensored by $\E_r$ we get
\begin{eqnarray*}
  \chi(\E_{r+1} \* \E_{r+1}^*) & = & \chi(\E_{r} \* \E_{r+1}^*)+ \chi(L_{\epsilon_{r+1}} \* \E_{r+1}^*) \\
  & = & \chi(\E_{r} \* \E_{r}^*) +\chi(\E_{r} \* L_{\epsilon_{r+1}}^*)+ \chi(L_{\epsilon_{r+1}} \* \E_{r+1}^*).
\end{eqnarray*}
Using the induction hypothesis, together with (iv) and (v) (with $r$ substituted by $r+1$), the right hand side can be written as
\begin{eqnarray*}
\left[-\frac{1}{2}\left(r^2-\epsilon_r\right)\left(h-1\right)+\epsilon_r\right]
+\left[-(h-1)\left\lfloor\frac{r+1}{2}\right\rfloor-\epsilon_r\right] +
\left[-(h-1)\left\lfloor\frac{r+2}{2}\right\rfloor+\epsilon_{r+1}h\right]
\end{eqnarray*}
An easy computation shows that this equals
\[ -\frac{1}{2}\left((r+1)^2-\epsilon_{r+1}\right)\left(h-1\right)+\epsilon_{r+1},\]
finishing the inductive step.

(vii) This is easily checked by induction again.
\end{proof}

We now define, for each $r \geq 1$, the scheme
$\UU(r)$ to be the modular family of the vector bundles $\E_r$ defined above.
For $r \geq 2$, the scheme $\UU(r)$ contains a subscheme $\UU(r)^{\rm ext}$ parametrizing bundles $\F_r$ that are non--split extensions of the form
\begin{equation} \label{eq:estensione}
  \xymatrix{
0 \ar[r] & \F_{r-1} \ar[r] & \F_r \ar[r] & L_{\epsilon_r} \ar[r] & 0, 
  }
\end{equation}
with $[\F_{r-1}] \in \UU(r-1)$.

\begin{lemma} \label{lemma:genUr}
  Let $\F_r$ be a general member of  $\UU(r)$. Then $\F_r$ is
  Ulrich of rank $r$ with slope $\mu:=L_0\cdot \xi_{m,n}=L_1\cdot \xi_{m,n}$.

  Moreover,
  \begin{itemize}
  \item[(i)] $\chi(\F_r \* \F_r^*)=-\frac{1}{2}\left(r^2-\epsilon_r\right)\left(h-1\right)+\epsilon_r$,
  \item[(ii)] $h^2(\F_r \* \F_r^*)=0$,
    \item[(iii)] $h^1(\F_r \* L_{\epsilon_{r+1}}^*) \leq \lfloor \frac{r+1}{2}\rfloor(h-1) +1$.
  \end{itemize}
\end{lemma}

\begin{proof}
  Ulrichness is an open property in the family, and the rank and slope are  constant,  so the general member of $\UU(r)$ is Ulrich
  of rank $r$ and slope $\mu$ as each $\E_r$ constructed above is (cf. Lemma \ref{lemma:2}(vii)).

  Properties (ii) and (iii) follows by specializing $\F_r$ to an $\E_r$ constructed above, and using semicontinuity and Lemma \ref{lemma:2}(iii) and (ii), respectively. Property (i) follows by Lemma \ref{lemma:2}(vi), since the given $\chi$ depends only on the Chern classes of the two factors and of $X_n$, which are constant in the family $\UU(r)$. 
\end{proof}

We wish to prove that the general member of $\UU(r)$ is slope--stable. To this end we will need a couple of auxiliary results.

\begin{lemma} \label{lemma:uniquedest}
  Let $r \geq 2$ and assume that $[\F_r] \in \UU(r)^{\rm ext}$ sits in a non--split sequence like \eqref{eq:estensione}
with $[\F_{r-1}] \in \UU(r-1)$ being slope--stable. Then 
  \begin{itemize}
  \item[(i)] $\F_r$ is simple (that is, $h^0(\F_r \* \F_r^*)=0$);
    \item[(ii)] if $\G$ is a destabilizing subsheaf of $\F_r$, then $\G^{*} \cong  \F_{r-1}^*$ and $(\F_r/\G)^* \cong L_{\epsilon_r}^*$; if furthermore $\F_r/\G$ is torsion--free, then  $\G \cong  \F_{r-1}$ and $\F_r/\G \cong L_{\epsilon_r}$. 
  \end{itemize}
\end{lemma}

\begin{proof}

  We first prove (ii).
  
Assume that $\G$ is a destabilizing subsheaf of $\F_r$, that is  $0<\rk(\G) < \rk(\F_r)=r$ and $\mu(\G) \geq \mu=\mu(\F_r)$. Define
  \[ \Q:=\im\{\G \subset \F_r \to L_{\epsilon_r}\} \; \; \mbox{and} \; \; \K:=\ker\{\G \to \Q\}.\]
  Then we may put \eqref{eq:estensione} into a commutative diagram with exact rows and columns:
  \[
    \xymatrix{ & 0 \ar[d] & 0 \ar[d] & 0 \ar[d] & \\
  0 \ar[r]   & \K \ar[d] \ar[r] & \G \ar[r] \ar[d] & \Q \ar[r] \ar[d] & 0 \\    
  0 \ar[r]   & \F_{r-1} \ar[r] \ar[d] & \F_r \ar[r] \ar[d] & L_{\epsilon_r} \ar[d] \ar[r] & 0 \\
  0 \ar[r]   & \K' \ar[r] \ar[d] & \F_r/\G \ar[r] \ar[d] & \Q' \ar[r] \ar[d] & 0 \\
  & 0  & 0  & 0  &
}
\]
defining $\K'$ and $\Q'$. We have $\rk(\Q) \leq 1$.

Assume that $\rk(\Q)=0$. Then $\Q =0$, whence $\K \cong \G$ and $\Q' \cong L_{\epsilon_r}$. Since $\mu(\K) =\mu(\G) \geq \mu=\mu(\F_{r-1})$ and $\F_{r-1}$ is slope--stable,
we must have
$\rk(\K)=\rk(\F_{r-1})=r-1$. It follows that $\rk(\K')=0$. Since
\[
  c_1(\K) =  c_1(\F_{r-1})-c_1(\K')=c_1(\F_{r-1})-D',\]
where $D'$ is an effective divisor supported on the codimension one locus of the support of $\K'$, we have 
\[ \mu \leq \mu(\K)=\frac{\left(c_1(\F_{r-1})-D'\right) \cdot \xi_{m,n}}{r-1}=
  \frac{c_1(\F_{r-1})\cdot \xi_{m,n}}{r-1}-\frac{D' \cdot \xi_{m,n}}{r-1}
= \mu-\frac{D' \cdot \xi_{m,n}}{r-1}.\]
  Hence $D'=0$, which means that $\K'$ is supported in codimension at least two.
  Thus, $\Shext^i(\K',\O_X)=0$ for $i \leq 1$, and it follows that $ \G^* \cong \K^* \cong \F_{r-1}^*$ and $(\F_r/\G)^* \cong {\Q'}^* \cong L_{\epsilon_r}^*$,  as desired. If furthermore $\F_r/\G$ is torsion--free, then we must have $\K'=0$, whence $\G \cong  \F_{r-1}$ and
$\F_r/\G \cong L_{\epsilon_r}$.

 Next we prove that $\rk(\Q)=1$ cannot happen. Indeed, if $\rk(\Q)=1$, then  $\rk(\K)=\rk(\G)-1 \leq r-2<r-1=\rk(\F_{r-1})$
and $\rk(\Q')=0$; in particular $\Q'$ is a torsion sheaf. Since
\[
  c_1(\K) =  c_1(\G)-c_1(\Q)=c_1(\G)-c_1(L_{\epsilon_r})+c_1(\Q')=c_1(\G)-c_1(L_{\epsilon_r})+D,
\]
where $D$ is an effective divisor supported on the codimension one locus of the support of $\Q'$, we have
\begin{eqnarray*}
  \mu(\K) & = & \frac{\Big(c_1(\G)-c_1(L_{\epsilon_r})+D\Big) \cdot \xi_{m,n}}{\rk(\K)} \geq \frac{\Big(c_1(\G)-c_1(L_{\epsilon_r})\Big) \cdot \xi_{m,n}}{\rk(\K)} 
\\
          & = & \frac{\mu(\G)\rk(\G)-c_1(L_{\epsilon_r})\cdot \xi_{m,n}}{\rk(\K)}
= \frac{\mu(\G)\rk(\G)-\mu}{\rk(\G)-1}
\geq
                \frac{\mu\rk(\G)-\mu}{\rk(\G)-1} =\mu
                \end{eqnarray*}
                This contradicts slope--stability of $\F_{r-1}$.

                To prove (i), assume that $\F_r$ is non--simple, that is, it admits a nontrivial endomorphism. By standard arguments, there exists a nonzero endomorphism $\varphi:\F_r \to \F_r$ dropping rank everywhere. (Take any endomorphism $\alpha$ that is not a constant times the identity, pick an eigenvalue $\lambda$ of $\alpha(x)$ for some $x \in X_n$ and set $\varphi=\alpha-\lambda \id$; then $\det(\varphi) \in H^0(\det(\F_r^*) \* \det(\F_r))=H^0(\O_{X_n})\cong \CC$ vanishes at $x$, whence  it  is identically zero.) Both $\ker(\varphi)$ and $\im (\varphi)$, being subsheaves of $\F_r$, are torsion--free, and one easily checks that at least one of them is destabiling. By part (ii), it follows that  either  $\ker(\varphi)\cong \F_{r-1}$ or $\im(\varphi)^* \cong \F_{r-1}^*$. In the first case, $\varphi$ factors through $L_{\epsilon_r}$, whence the map $\F_r \to L_{\epsilon_r}$ in \eqref{eq:estensione} splits, a contradiction. In the second case, the natural injection $\im(\varphi) \subset \im(\varphi)^{**} \cong \F_{r-1}^{**} \cong \F_{r-1}$ shows that $\varphi$ factors through $\F_{r-1}$, whence the map $\F_{r-1} \to \F_r$ in \eqref{eq:estensione} splits, again a contradiction.
\end{proof}

\begin{lemma} \label{lemma:dimU}
  Let $r \geq 2$ and assume that the general member of $\UU(r-1)$ is slope--stable. Then $\UU(r)$ is  generically  smooth of dimension $\frac{1}{2}(r^2-\epsilon_r)(h-1)+\epsilon_{r+1}$ and properly contains $\UU(r)^{\rm ext}$. 
\end{lemma}

\begin{proof}
  The general member $\F_r$ of $\UU(r)$ satisfies $h^0(\F_r \* \F_r^*)=1$ by Lemma \ref{lemma:uniquedest}(i) and $h^2(\F_r \* \F_r^*)=0$ by Lemma \ref{lemma:genUr}(ii). Hence  one has (see, e.g., \cite[Prop. 2.10]{ch})  that  $\UU(r)$ is  generically  smooth of dimension
  \begin{eqnarray*}
    h^1(\F_r \* \F_r^*) & = & -\chi(\F_r \* \F_r^*)+h^0(\F_r \* \F_r^*)+h^2(\F_r \* \F_r^*)=-\chi(\F_r \* \F_r^*)+1 \\
                        & = & \frac{1}{2}\left(r^2-\epsilon_r\right)\left(h-1\right)-\epsilon_r+1 =\frac{1}{2}\left(r^2-\epsilon_r\right)\left(h-1\right)+\epsilon_{r+1},
  \end{eqnarray*}
where we have used Lemma \ref{lemma:genUr}(i),  as claimed.

  Similarly, being slope-stable, also the general member $\F_{r-1}$ of $\UU(r-1)$ satisfies $h^0(\F_{r-1} \* \F_{r-1}^*)=1$, so the same reasoning shows that
  \begin{equation}
    \label{eq:dimUr-1}
\dim (\UU(r-1))=    \frac{1}{2}\left((r-1)^2-\epsilon_{r-1}\right)\left(h-1\right)+\epsilon_{r}.
  \end{equation}
  Morover,
  \begin{equation}
    \label{eq:dimextv}
    \dim (\Ext^1(L_{\epsilon_r},\F_{r-1}))=h^1(\F_{r-1} \*L_{\epsilon_r}^*)\leq \left\lfloor \frac{r}{2}\right\rfloor(h-1) +1.
  \end{equation}
  by Lemma \ref{lemma:genUr}(iii). Hence
  \begin{eqnarray*}
    \dim(\UU(r)^{\rm ext}) &  \leqslant  & \dim (\UU(r-1))+\dim \PP(\Ext^1(L_{\epsilon_r},\F_{r-1})) \\
                       & \leq & \frac{1}{2}\left((r-1)^2-\epsilon_{r-1}\right)\left(h-1\right)+\epsilon_{r} +\left\lfloor \frac{r}{2}\right\rfloor(h-1) \\
                       & = & \frac{1}{2} \left(r^2-\epsilon_{r}\right)\left(h-1\right)+\epsilon_{r+1}-\left(\left\lfloor \frac{r+1}{2} \right\rfloor-\epsilon_r\right)(h-1)+\epsilon_r-\epsilon_{r+1} \\
     & = & \dim(\UU(r)) - \left(\left\lfloor \frac{r+1}{2} \right\rfloor-\epsilon_r\right)(h-1)+\epsilon_r-\epsilon_{r+1},
    \end{eqnarray*}
  and one easily sees that this is strictly less than
  $\dim(\UU(r))$, since $r \geq 2$ and $h \geq 3$.  Thus, $\UU(r)^{\rm ext}$ is properly contained in $\UU(r)$, as claimed.
\end{proof}

We can now prove slope--stability of the general member of $\UU(r)$.

\begin{proposition}
  Let $r\geq 1$. The general member of $\UU(r)$ is slope--stable.
\end{proposition}

\begin{proof}
  We use induction on $r$, the result being trivially true for $r=1$. 

 Assume $r \geq 2$ and that the general member of $\UU(r)$ is not slope--stable, whereas the general member of $\UU(r-1)$ is. Then we may find a one-parameter family of bundles $\{\F^{(t)}\}$ over the disc $\DD$ such that $\F^{(t)}$ is a general member of $\UU(r)$ for $t \neq 0$ and $\F^{(0)}$ lies in $\UU(r)^{\rm ext}$, and such that we have a destabilizing sequence
\begin{equation} \label{eq:destat1} \xymatrix{
    0 \ar[r] & \G^{(t)} \ar[r] & \F^{(t)} \ar[r] & \Q^{(t)} \ar[r] & 0}
  \end{equation}
  for $t \neq 0$, which we can take to be saturated, that is, such that $\Q^{(t)}$ is torsion free, whence so that $\G^{(t)}$ and $\Q^{(t)}$ are (Ulrich) vector bundles  (see \cite[Thm. 2.9]{ch} or \cite[(3.2)]{be2}).

  The limit of $\PP(\Q^{(t)}) \subset \PP(\F^{(t)})$ defines a subvariety of $\PP(\F^{(0)})$ of the same dimension as  $\PP(\Q^{(t)})$,   whence a coherent sheaf $\Q^{(0)}$ of rank $\rk(\Q^{(t)})$ with a surjection $\F^{(0)} \to \Q^{(0)}$. Denoting by $\G^{(0)}$ its kernel, we have
$\rk(\G^{(0)})=\rk(\G^{(t)})$ and $c_1(\G^{(0)})=c_1(\G^{(t)})$. Hence, \eqref{eq:destat1} specializes to a destabilizing sequence for $t=0$. 
Lemma \ref{lemma:uniquedest} yields  that ${\G^{(0)}}^*$ (respectively,
${\Q^{(0)}}^*$) 
is the dual of a member of  $\mathfrak{U}(r-1)$ 	 (resp., the dual of $L_{\epsilon_r}$).
It follows that
${\G^{(t)}}^*$ (resp., ${\Q^{(t)}}^*$)
is a deformation of the dual of a member of  $\mathfrak{U}(r-1)$  (resp., a deformation of $L_{\epsilon_r}^*$), whence that $\G^{(t)}$ is a deformation of a member of  $\mathfrak{U}(r-1)$, as both are locally free, and $\Q^{(t)} \cong L_{\epsilon_r}$, for the same reason. 

In other words, the general member of $\UU(r)$ 
is an extension of $L_{\epsilon_r}$ by a member of
 $\mathfrak{U}(r-1)$. Hence $\UU(r)=\UU(r)^{\rm ext}$, contradicting Lemma \ref{lemma:dimU}. 
\end{proof}

We have therefore proved:

\begin{thm} \label{thm:higher}
  For any $r \geq 1$, the blown--up plane $X_n$,  with $n\geq 2$, carries slope--stable rank--$r$ Ulrich bundles, and their moduli space  contains a reduced and irreducible component  of dimension $\frac{1}{2}(r^2-\epsilon_r)(h-1)+\epsilon_{r+1}$.
\end{thm}

\end{document}